\documentclass{amsart}
\usepackage{mathrsfs}
\usepackage{amsfonts}
\usepackage{amssymb}
\usepackage{amsxtra}
\usepackage{dsfont}
\usepackage{color}
\usepackage[english,polish]{babel}
\usepackage[colorlinks,citecolor=blue,urlcolor=blue,bookmarks=true]{hyperref}

\hypersetup{
pdfpagemode=UseNone,
pdfstartview=FitH,
pdfdisplaydoctitle=true,
pdfborder={0 0 0}, 
pdftitle={Limit theorems for random walks},
pdfauthor={Alexander Bendikov, Wojciech Cygan and Bartosz Trojan},
pdflang=en-GB
}

\newcommand{\DD}{\mathbf{D}}
\newcommand{\CC}{\mathbb{C}}
\newcommand{\ZZ}{\mathbb{Z}}
\newcommand{\NN}{\mathbb{N}}
\newcommand{\RR}{\mathbb{R}}
\newcommand{\PP}{\mathbb{P}}
\newcommand{\EE}{\mathbb{E}}

\newcommand{\calL}{\mathcal{L}}
\newcommand{\calS}{\mathcal{S}}
\newcommand{\calF}{\mathcal{F}}
\newcommand{\calO}{\mathcal{O}}

\newcommand{\calV}{\mathcal{V}}
\newcommand{\calT}{\mathcal{T}}
\newcommand{\calI}{\mathcal{I}}

\newcommand{\dth}{{\: \rm d}\theta}

\newcommand{\pl}[1]{\foreignlanguage{polish}{#1}}

\newcommand{\norm}[1]{\lVert {#1} \rVert}
\newcommand{\abs}[1]{\lvert {#1} \rvert}
\newcommand{\sprod}[2]{\langle {#1}, {#2} \rangle}

\newcommand{\Log}{\operatorname{Log}}

\theoremstyle{plain}
\newtheorem{theorem}{Theorem}[section]

\newtheorem{lemma}[theorem]{Lemma}
\newtheorem{claim}{Claim}
\newtheorem{corollary}[theorem]{Corollary}

\theoremstyle{plain}
\newcounter{thm}

\newtheorem{main_theorem}[thm]{Theorem}

\theoremstyle{remark}
\newtheorem{remark}{Remark}

\title[Limit theorems for random walks]
{Limit theorems for random walks}

\author{Alexander Bendikov}
\thanks{Research of A.~Bendikov was supported by National Science Centre, Poland, Grant DEC-2012/05/B/ST1/00613}
\address{
	Alexander Bendikov\\
	Instytut Matematyczny\\
	Uniwersytet \pl{Wroc{\lll}awski}\\
	Pl. Grun\-waldzki 2/4\\
	50-384 \pl{Wroc{\lll}aw}\\
	Poland}
\email{bendikov@math.uni.wroc.pl}

\author{Wojciech Cygan}
\thanks{Research of W.~Cygan was supported by National Science Centre, Poland, Grant DEC-2013/11/N/ST1/03605 and by
	Austrian Science Fund project FWF P24028}
\address{
	Wojciech Cygan\\
	Instytut Matematyczny\\
	Uniwersytet \pl{Wroc{\lll}awski}\\
	Pl. Grun\-waldzki 2/4\\
	50-384 \pl{Wroc{\lll}aw}\\
	Poland}
\email{cygan@math.uni.wroc.pl}

\author{Bartosz Trojan}
\address{
	\pl{
	Bartosz Trojan\\
	Wydzia\l{} Matematyki,
	Politechnika Wroc\l{}awska\\
	Wyb. Wyspia\'{n}skiego 27\\
	50-370 Wroc\l{}aw\\
	Poland}
}
\email{bartosz.trojan@pwr.edu.pl}

\makeatother

\numberwithin{equation}{section}

\begin{document}
\selectlanguage{english}

\begin{abstract}
We consider a random walk $S_{\tau}$ which is obtained from the simple
random walk $S$ by a discrete time version of Bochner's subordination. We prove that under certain
conditions on the subordinator $\tau$ appropriately scaled random walk $S_{\tau}$ converges in the Skorohod
space to the symmetric $\alpha$-stable process $B^{\alpha}$. We also prove asymptotic formula for the
transition function of $S_{\tau}$ similar to the P\'{o}lya's asymptotic formula for $B^{\alpha}$.
\end{abstract}

\maketitle

\section{Introduction}
It is well known and due to P\'{o}lya \cite{polya} that the transition density $p_{\alpha }(x, t)$, 
$0 < \alpha < 2$, of the one-dimensional symmetric $\alpha $-stable process has the following asymptotic expression
\footnote{We write $f\sim g$ at $a$ if $f(x)/g(x)\to 1$, as $x\to a$.}
\begin{equation}
	\label{Polya}
	p_{\alpha }(x, t)
	\sim
	c_\alpha t |x|^{-1-\alpha},\quad 
	\mathrm{as}
	\ t |x|^{-\alpha} \rightarrow 0
\end{equation}
where 
\begin{align}
	\label{eq:25}
	c_\alpha
	=
	\alpha 2^{\alpha -1}
	\pi ^{-3/2}
	\Gamma \left(\frac{\alpha+ 1}{2}\right)\Gamma \left(\frac{\alpha}{2}\right)
	\sin \Big(\frac{\pi \alpha}{2} \Big).
\end{align}
For multidimensional symmetric $\alpha$-stable process similar asymptotic formula has been found by
Blumenthal and Getoor \cite{blumen}. Both results were obtained by using the method of characteristic functions
and the scaling property of $\alpha $-stable distributions. An elegant proof of the P\'{o}lya--Blumenthal--Getoor
asymptotic formula was proposed by Bendikov \cite{b}. Bendikov used the fact that the symmetric 
$\alpha$-stable process can be obtained from the standard Brownian motion by means of the subordination, the notion
introduced in the theory of Markov semigroups by Bochner \cite{bochner}. 

From the probabilistic point of view, the $\alpha$-stable process $B^\alpha = (B^\alpha_{t} : t \geq 0)$,
$0 < \alpha < 2$, is obtained from the Brownian motion $B = (B_{t}: t \geq 0)$ by setting
$B^\alpha _{t}=B_{\sigma ^\alpha _{t}}$ where the \emph{subordinator} $\sigma^\alpha = (\sigma ^\alpha _{s} : s \geq 0)$
is independent of $B$ non-decreasing L\'{e}vy process such that 
\[
	\EE \big( e^{-\lambda \sigma _s^\alpha}\big) = e^{-s\lambda ^{\alpha/2}},\quad \lambda \geq 0.
\]
From the analytical point of view, the transition density $p_\alpha (x, t)$ of
$B^\alpha $ is obtained as a time average of the transition density $%
p(x, t)$ of $B$, i.e. 
\begin{equation*}
	p_\alpha (x, t) =\int_{[0, \infty)} p(x, s) {\: \rm d} \nu ^\alpha _{t}(s)
\end{equation*}
where $\nu ^\alpha _{t}$ is the distribution function of the random
variable $\sigma ^\alpha _{t}$. 
In particular, the minus infinitesimal generator of the process $B^\alpha $ is $(-\Delta)^{\alpha /2}$
where $\Delta$ is the classical Laplace operator.

In this article we study a certain class of random walks obtained from finite range random walks, in particular the
simple random walk, by means of discrete time subordination recently developed in \cite{bs}. We prove that
any random walk from our class, being appropriately scaled, converges in the Skorohod space to the $\alpha$-stable
process $B^\alpha$. We also prove that a discrete space and time counterpart of the P\'{o}lya's asymptotic formula
\eqref{Polya} holds true.

To illustrate the main results of the paper we consider $S$ to be the simple random walk in $\ZZ ^d$,
a discrete space and time counterpart of the Brownian motion in $\RR^d$. Let $0 < \alpha < 2$ and $S^\alpha$
be the subordinate random walk defined via the functional calculus of (minus) generators of discrete time Markov
semigroups
\begin{align*}
	I - P_\alpha =(I-P)^{\alpha /2}.
\end{align*}
Here $P$ and $P_\alpha$ are the transition operators of $S$ and $S^\alpha$, respectively. In other words,
$S^\alpha$ coincides in distribution with $S_{\tau ^\alpha }$ where the discrete subordinator $\tau ^\alpha $
is uniquely determined by the formula 
\begin{align*}
	\EE \big( e^{-\lambda \tau ^\alpha _{n}}\big)
	=\left(1-(1-e^{-\lambda})^{\alpha /2}\right)^n,\quad \lambda \geq 0.
\end{align*} 
In Sections \ref{sec:FLT} and
\ref{sec:Asymp} we prove the following two theorems (or even more general versions of them).
\begin{main_theorem}
	\label{thm:5}
	For each $0 < \alpha < 2$, the sequence of random elements 
	\[
		\big(n^{-1/\alpha} S^\alpha_{[nt]} : t \geq 0\big)
	\] 
	converges in the Skorohod space to $B^\alpha$.
\end{main_theorem}
Let $\{e_1, \ldots, e_d\}$ be the standard basis for $\RR^d$ and $\norm{\:\cdot\:}_2$ the Euclidean norm.
\begin{main_theorem}
	\label{thm:3}
	Let $0 < \alpha < 2$ and $p_\alpha (x, n)$ be the transition function of $S^\alpha$, 
	then as $n$ and $\norm{x}_2$ tend to infinity we have
	\begin{enumerate}
	\item if $n \norm{x}_2^{- \alpha}$ tends to zero then
	\begin{align*}
		\frac{1}{2d} \sum_{j = 1}^d \big(p_\alpha(x + e_j, n) + 2 p_\alpha(x, n) + p_\alpha(x - e_j, n)\big)
		\sim
		C_{d, \alpha}
		n \norm{x}_2^{-d - \alpha};
	\end{align*}
	\item if $n \norm{x}_2^{-\alpha}$ tends to infinity then
	\begin{align*}
		p_\alpha(x, n) \sim p_\alpha(0, n) \sim D_{d, \alpha} n^{-d/\alpha}.
	\end{align*}
	\end{enumerate}
\end{main_theorem}
	The constants in Theorem \ref{thm:3} are given explicitly by the formulas
	\begin{align*}
		C_{d,\alpha} = 
		\alpha
		2^{\alpha/2} \pi^{-d/2-1} 
		d^{-\alpha /2}\, \Gamma\Big(\frac{\alpha}{2}\Big) %
		\Gamma\Big(\frac{d + \alpha}{2}\Big) 
		\sin \Big(\frac{\pi \alpha}{2}\Big)
	\end{align*}
	and
	\[
		D_{d, \alpha} = (2 \pi)^{d/2} \frac{\Gamma(1 + d/\alpha)}{\Gamma(1 + d/2)}.
	\]
	Let us observe that $C_{d, \alpha}$ has the factor $d^{-\alpha/2}$ which is not present in \eqref{eq:25}.
	This is a consequence of the asymptotic behaviour of the simple random walk on $\ZZ^d$.

The main ingredients in our proofs are the continuity principle of
superposition in the Skorohod space due to Whitt \cite{Whitt}, and
the asymptotic formulas for the discrete time subordinators. In Section \ref{sec:3} we prove the following theorem.
\begin{main_theorem}
	\label{thm:6}
	Let $F_n(t) = \PP(\tau_n \leq t)$. Whenever $n t^{-\alpha}$ tends to zero,
	\[
		1 - F_n(t) \sim \frac{n t^{-\alpha}}{\Gamma(1 - \alpha/2)}.
	\]
\end{main_theorem}
Let us mention that Theorem \ref{thm:6} can be regarded as a version of the global large deviation result for random
walks with increments belonging to the domain of attraction of a stable law, see e.g. \cite{doney, nag} and references
therein.

In this work we consider more general subordinators. Namely, for any Bernstein function
$\psi$ which varies regularly at zero we construct related subordinator $\tau$ and random walk $S_\tau$,
and we establish $\tau$-versions of Theorems \ref{thm:5}, \ref{thm:3} and \ref{thm:6}.

Theorem \ref{thm:5} (as well as the more general version in Theorem \ref{FLT}) was first obtained by A.~Mimica in 
\cite{mim}. Furthermore, he also proved the converse of the theorem and treats the case $\alpha = 2$. While Mimica
applied the method of characteristic functions, our approach is different and uses superposition in the Skorohod space.

\section{Preliminaries}
\subsection{Finite range random walks}
Let $S=(S_n : n \in \mathbb{N})$ be a random walk driven by a probability measure $p$ and let 
$p^{(n)} = p(\: \cdot \:, n)$ be its $n$-fold convolution. We say that $(S,p)\in \mathscr{S}$ if the following
properties hold
\begin{enumerate}
	\item $p$ is supported by a finite set $\mathcal{V}\subset \ZZ ^d$;
	\item for each $x \in \ZZ^d$ there is $n \in \NN$ such that $p(x, n) > 0$;
	\item $\sum_{v \in \ZZ^d} p(v) \cdot v=0$.
\end{enumerate}
Generally, a random walk from the class $\mathscr{S}$ is not aperiodic. Thus, we let $r$ to be its period, that is 
the greatest common divisor of all $n$ such that $p(0, n) > 0$.
Then the space $\ZZ^d$ decomposes into $r$ disjoint classes $R_0, \ldots, R_{r-1}$ where
\[
	R_j = \big\{x \in \ZZ^d : p(x, j + kr) > 0 \text{ for some } k \geq 0\big\}.
\]
For instance, if $S$ is the simple random walk starting at the origin, we have $r=2$ and $R_0$
is the class of points which $S$ visits in even number of steps.

Before we state the asymptotic formula for $p(x, n)$ we need to introduce a quadratic form on $\RR^d$ defined by
\begin{equation}
	\label{eq:41}
	\sprod{Q u}{u} = \sum_{v \in \ZZ^d} p(v) \sprod{v}{u}^2.
\end{equation}
We note that for the simple random walk $Q = d^{-1} I$ where $I$ is the identity matrix. Let $Q^{-1}$ be the inverse
of $Q$. We also fix a norm on $\RR^d$, $\norm{x}^2 = \sprod{Q^{-1} x}{x}$. By \cite[Corollary 3.2 and Remark 1]{tr1},
we obtain that for every $j \in \{0, \ldots, r-1\}$, $n \in \NN$ and $x \in R_j$ such that $\norm{x} \leq n^{2/3}$,
if $n \equiv j \pmod r$ then
\begin{equation}
	\label{eq:35}
	p(x, n) = 
	(2\pi n)^{-\frac{d}{2}}
	(\det Q)^{-\frac{1}{2}}
	e^{-\frac{1}{2n} \sprod{Q^{-1} x}{x}}
	\Big(r + \calO\big(n^{-1}\big) + \calO\big(n^{-2} \norm{x}^3\big)\Big)
\end{equation}
otherwise $p(x, n) = 0$. In particular, we obtain
\begin{corollary}
	\label{cor:1}
	For every $\epsilon > 0$, and $x \in \ZZ^d$,
	\[
		\sum_{j = 0}^{r-1}
		p(x, n+j) = (2\pi n)^{-\frac{d}{2}}(\det Q)^{-\frac{1}{2}} e^{-\frac{1}{2n}\sprod{Q^{-1} x}{x}}
		\Big(r + \calO\big(n^{-1}\big) + \calO\big(n^{-2} \norm{x}^3\big)\Big)
	\]
	provided $\norm{x} \leq n^{2/3}$.
\end{corollary}
Let us comment about the asymptotic of $p(x, n)$. For aperiodic random walks with mean zero to get \eqref{eq:35}
one may use \cite[Theorem 2.3.11]{law}. A more refined asymptotic is recently developed by Trojan in \cite{tr1}.
It is valid in a larger region and applicable to random walks with non-zero mean and not necessary aperiodic.

\subsection{Discrete time subordination}
\label{sub:3}
For continuous time Markov processes, subordination is a useful procedure of obtaining new Markov processes.
The later may differ very much from the original one but its properties can be understood in terms
of the old process. For example, in this way the symmetric stable process can be obtained from the Brownian motion
(see e.g. \cite{b}).

From a probabilistic point of view, a new process $(Y_{t} : t \geq 0)$ is obtained from the 
process $(X_{t} : t \geq 0)$ by setting $Y_{t}=X_{\sigma _{t}}$ where the \emph{subordinator}
$(\sigma _{s} : s \geq 0)$ is a non-decreasing L\'{e}vy process taking values in $(0, \infty)$ and
independent of $(X_{t} : t \geq 0)$ (see e.g. \cite[Section X.7]{feller}). From an analytical point of
view, the transition function $h_{\sigma}(x,B, t)$ of the new process is obtained as a time
average of the transition function $h(x, B, t)$ of the original one, i.e.
\begin{equation*}
	h_{\sigma }(x, B, t)
	=\int_{[0, \infty)} h(x, B, s) {\: \rm d} \nu_{t}(s)
\end{equation*}
where $\nu _{t}$ is the distribution function of the random variable $\sigma _{t}$.
Subordination was introduced by Bochner in the context of
semigroup theory (see \cite[footnote, p. 347]{feller}). Formally, the minus infinitesimal
generator $\mathcal{B}$ of the process $(Y_{t} : t \geq 0)$ is a function of the minus
infinitesimal generator $\mathcal{A}$ of the process $(X_{t} : t \geq 0)$, that is, 
$\mathcal{B=\psi (A)}$ for some function $\psi$.

A discrete time version of subordination, where the functional calculus equation 
$\mathcal{B}=\psi (\mathcal{A})$ serves as the defining starting point, has been
considered by Bendikov and Saloff-Coste in \cite{bs}. To describe the procedure, let $R$ be the operator
of convolution by $p$, a probability density on $\ZZ^d$. Then $L=I-R$, where $I$ is the identity operator, can be
considered as minus the Markov generator of the associated random walk. For a proper function $\psi$ we may define a
\emph{subordinate random walk} $S^\psi$ as the random walk process with the generator
$-\psi (L)$. As it has been shown in \cite{bs}, the appropriate class consists of Bernstein functions.

Let us recall that a smooth and non-negative function $\psi: [0, \infty) \rightarrow [0, \infty)$ is called 
a Bernstein function if for all $n \in \NN$ and $x > 0$
$$
(-1)^n \psi^{(n)}(x) \leq 0.
$$
For a Bernstein function $\psi$ there are constants $a, b \geq 0$ and a measure $\nu$ on
$[0, \infty)$ satisfying 
$$
\int_{[0, \infty)} \min\{1, s\} {\: \rm d}\nu(s) < \infty,
$$
and such that
\begin{equation}
	\label{eq:40}
	\psi(x) = a + b x + \int_{[0, \infty)} (1 - e^{-x s}) {\: \rm d}\nu(s).
\end{equation}
Given a Bernstein function $\psi$ with $\psi(0) = 0$ and $\psi(1) = 1$ for $k \in \NN$ we set
\begin{equation}
	\label{eq:3}
	c(\psi, k) = \frac{1}{k!}\int_{[0, \infty)} t^k e^{-t} {\: \rm d}\nu(t) + 
	\begin{cases}
		b & \text{ if } k = 1,\\
		0 & \text{ otherwise}.
	\end{cases}
\end{equation}
Then the random walk $S^\psi$ with the generator $-\psi(L)$ has the probability density
\begin{equation}
	\label{eq:1}
	p_\psi(x) = \sum_{k \geq 1} p(x, k) c(\psi, k),
\end{equation}
see \cite[Proposition 2.3]{bs}. Formula \eqref{eq:1} has the following probabilistic
interpretation. Let $(R_k : k \in \NN)$ be
a sequence of independent and identically distributed integer-valued random variables also independent
of $S$ and such that $\mathbb{P} (R_i=k)=c(\psi, k)$. Then for $\tau _n=R_1+\ldots +R_n$ we have
\begin{equation}
	\label{eq:4}
	\mathbb{P} (\tau _n=k)
	=\sum _{k_1+\ldots +k_n=k} 
	c(\psi, k_1) \cdots c(\psi, k_n).
\end{equation}
In particular, the law of $S^\psi_n$ is the same as the law of $S_{\tau_n}$ and is given by
\begin{equation}
	\label{eq:6}
	p_\psi(x, n)
	=
	\sum _{k\geq n} p(x, k)
	\mathbb{P} (\tau _n=k).
\end{equation}

\subsection{Asymptotics related to subordinators}
\label{sec:3}
For a given $(S, p) \in \mathscr{S}$ and a Bernstein function $\psi$, the relations \eqref{eq:4} and \eqref{eq:6}
defining the transition function $p_\psi (x, n)$ of the subordinate random walk $S^\psi$ are rather complicated.
In the course of study, we are going to restrict ourselves by considering Bernstein functions $\psi$
\emph{regularly varying} at zero of index $0 < \alpha/2 < 1$, such that $\psi(0) = 0$ and $\psi(1)=1$.

Let us recall that a function $f$ defined in an interval $(0, a)$, $a>0$, is \emph{regularly varying} of index $\beta $
at zero if $f(x)= x^\beta \ell (1/x) $ where $\ell$ satisfies
\begin{equation*}
	\lim_{x \to \infty} \frac{\ell(\lambda x)}{\ell(x)} = 1,
\end{equation*}
for each $\lambda >0$. The function $\ell$ is called \emph{slowly varying} at infinity. For the detailed exposition
of the theory of regular variation we refer the reader to \cite{bgt}. Here, we only mention \emph{Potter bounds}
(see \cite{pot}, see also \cite[Theorem 1.5.6]{bgt}). If $\ell$ varies slowly at infinity then for
every $\epsilon > 0$ and $C > 1$ there exists $x_0 > 0$ such that for all $x ,y \geq x_0$ 
\begin{equation} 
	\label{eq:22}
	\ell(y) \leq C\, \ell(x) \max\{y/x, x/y\}^\epsilon.
\end{equation}
The main result of this section is the following theorem. 
\begin{theorem}\label{thm:2} 
	Set $F_n(t)=\mathbb{P} (\tau _n \leq t)$. If $n$ and $t$ both tend to infinity in such a way that
	$n\psi\big(t^{-1}\big)$ tends to zero then
	\begin{align*}
		1-F_n(t) \sim \frac{n \psi\big(t^{-1}\big)}{%
		\Gamma(1-\alpha/2)} .
	\end{align*}
\end{theorem}
\begin{proof}
	The statement of the theorem may be expressed as follows: for any sequence $(t_n : n \in \NN)$ tending to
	infinity and satisfying 
	\begin{equation}
		\label{eq:2}
		\lim_{n \to \infty} n \psi\big(t_n^{-1}\big) = 0,
	\end{equation}
	one has
	\begin{equation}
		\label{eq:11}
		\lim_{n \to \infty} \frac{1-F_n(t_n)}{n \psi\big(t_n^{-1}\big)} =\frac{1}{\Gamma(1-\alpha/2)}.
	\end{equation}
	Let us observe that the Laplace transform of the function $F_1$ is equal to
	\begin{equation}
		\label{eq:23} 
		\calL\{F_1\}(\lambda) = \int_0^\infty F_1(y) e^{-\lambda y} {\: \rm d}y
		=
		\lambda^{-1} \big(1 - \psi\big(1 - e^{-\lambda} \big)\big).
	\end{equation}
	Hence, for any $n \in \NN$ 
	\begin{equation}
		\label{eq:18}
		\calL\{1 - F_n\}(\lambda) 
		= \lambda^{-1} \big(1 - \big(1 - \psi\big(1 - e^{-\lambda}\big)\big)^n\big).
	\end{equation}
	Next, we define
	\begin{equation*}
		\calF_n(x) = \int_0^x 1 - F_n(y) {\: \rm d} y,
	\end{equation*}
	for $x > 0$. Then, by \eqref{eq:18}
	\begin{align}
		\calF_n(x) 
		= \int_0^x {\: \rm d}\calF_n(y) &\leq e \int_0^x e^{-y/x} {\: \rm d} \calF_n(y) \nonumber\\
		& \leq e \calL\{{\rm d} \calF_n\}(1/x) \nonumber\\ 
		& \leq C x n \psi\big(1 - e^{-1/x}\big).\label{eq:7}
	\end{align}
	In what follows $C$ denotes an absolute positive constant which may vary form line to line.

	For a given sequence $(t_n : n \in \NN)$ satisfying \eqref{eq:2} we define a sequence of tempered
	distribution on $[0,\infty )$ by setting
	\begin{equation*}
		\Lambda_n(f) = \frac{1}{n t_n \psi\big(t_n^{-1}\big)}
		\int _0^\infty f(x) \calF_n(t_n x){\: \rm d}x
	\end{equation*}
	for any $f \in \calS\big([0,\infty )\big)$. Let us recall that the space $\mathcal{S}\big([0,\infty )\big)$
	consists of Schwartz functions on $\mathbb{R}$ restricted to $[0,\infty )$ and
	$\calS^\prime \big([0,\infty )\big)$ consists of tempered distributions supported by $[0,\infty )$,
	see \cite{vlad} for details.

	We claim that there is $C > 0$ such that for any $f \in \calS\big([0,\infty )\big)$ and $n \in \NN$
	\begin{equation} 
		\label{eq:8}
		\abs{\Lambda_n (f)} 
		\leq C \sup _{ x \geq 0} \big\{(1 + x)^3 {\lvert {f(x)} \rvert} \big\}.
	\end{equation}
	Indeed, by \eqref{eq:7} we have
	\begin{equation*}
		|\Lambda_n (f)|
		\leq 
		\frac{C}{\psi\big(t_n^{-1}\big)} 
		\int_0^\infty x {\lvert {f(x)} \rvert} \psi\big(1 - e^{-1/(t_n x)}\big) {\: \rm d}x.
	\end{equation*}
	To bound the right-hand side, let us recall that $\psi(\lambda) = \lambda^{\alpha/2} \ell(1 / \lambda)$ for some
	$0 < \alpha < 2$. By \eqref{eq:22}, there are $N \geq 1$ and $C > 0$ such that for $n \geq N$
	\[
		\psi\big(t_n^{-1}\big) \geq C t_n^{-(\alpha/2+1)/2}.
	\]
	Thus, we can estimate 
	\begin{align*}
		\frac{1}{\psi\big(t_n^{-1}\big)} 
		\int_0^{t_n^{-1/2}} x {\lvert {f(x)} \rvert} \psi\big(1 - e^{-1/(t_n x)}\big) {\: \rm d} x
		& \leq C \sup_{x \geq 0} \big\{{\lvert {f(x)} \rvert} \big\} t_n^{(\alpha/2+1)/2} 
		\int_0^{t_n^{-1/2}} x {\: \rm d}x \\
		& \leq C \sup _{ x \geq 0}\big\{{\lvert {f(x)} \rvert} \big\}. 
	\end{align*}
	Let $x_n = - 1 / \big(t_n \cdot \log \big(1 - 1/t_n\big)\big)$ and 
	\begin{equation*}
		A_n(x) = t_n \big(1 - e^{-1/(t_n x)}\big).
	\end{equation*}
	For $t_n^{-1/2} \leq x \leq x_n$ we have 
	\[
		1 \leq A_n(x) \leq t_n^{1/2}.
	\]
	Again, by \eqref{eq:22}, there is $N \geq 1$ such that for $n \geq N$ 
	\[
		\ell\big(A_n(x)^{-1} t_n \big) \leq \ell(t_n) A_n(x)^{(1-\alpha/2)/2},
	\]
	thus
	\begin{equation*}
		\frac{\psi\big(A_n(x)\cdot t_n^{-1}\big)} {A_n(x) \psi\big(t_n^{-1}\big)}
		\leq A_n(x)^{(\alpha/2-1)/2} \leq 1.
	\end{equation*}
	Since $x A_n(x) \leq 1$ whenever $x > 0$, we obtain 
	\begin{align*}
		\frac{1}{\psi\big(t_n^{-1}\big)} \int_{t_n^{-1/2}}^{x_n} x 
		{\lvert {f(x)} \rvert} \psi\big(A_n(x) \cdot t_n^{-1}\big) {\: \rm d}x 
		&\leq \int_{t_n^{-1/2}}^{x_n} x {\lvert {f(x)} \rvert} A_n(x) {\: \rm d}x\\
		&\leq C \sup _{x\geq 0} \big\{ (1 + x)^2 {\lvert {f(x)} \rvert}  \big\}.
	\end{align*}
	Finally, if $x\geq x_n$ then $A_n(x) \leq 1$. Therefore, by monotonicity we get 
	\begin{equation*}
		\psi\big(A_n(x) \cdot t_n^{-1}\big) \leq \psi\big(t_n^{-1}\big).
	\end{equation*}
	Hence, 
	\begin{align*}
		\frac{1}{\psi\big(t_n^{-1}\big)}
		\int_{x_n}^\infty x {\lvert {f(x)} \rvert}
		\psi\big(A_n(x) \cdot t_n^{-1}\big) {\: \rm d}x 
		& \leq \int_{x_n}^\infty x {\lvert {f(x)} \rvert} {\: \rm d}x \\
		& \leq C \sup _{x\geq 0} \big\{ (1 + x)^3 {\lvert {f(x)} \rvert} \big\},
	\end{align*}
	which finishes the proof of the claim \eqref{eq:8}.

	Let us observe, that \eqref{eq:8} implies that the family of distributions $(\Lambda_n : n \in \NN)$ is
	equicontinuous. Next, we calculate $\Lambda_n(f_\tau)$ for $f_\tau(x) = e^{-\tau x}$, $\tau > 0$. 
	The integration by parts yields
	\begin{equation}
		\label{eq:9}
		\Lambda_n (f_{\tau}) 
		= \frac{1}{n t_n\psi\big(t_n^{-1}\big)}
		\int_0^\infty e^{-\tau x} \calF_n(t_n x) {\: \rm d}x
		= \frac{1}{n \tau t_n \psi\big(t_n^{-1} \big)} \calL\{{\rm d}\calF_n\}(\tau/t_n),
	\end{equation}
	We need the following observation. 
	\begin{claim}
		\label{cl:1}
		Let $(a_n : n \in \NN)$ and $(b_n : n \in \NN)$ be any two sequences
		of positive numbers both tending to zero and such that
		\[
			\lim_{n \to \infty} \frac{a_n}{b_n} = 1.
		\]
		Then
		\[
			\lim_{n \to \infty} \frac{\psi(a_n)}{\psi(b_n)} = 1.
		\]
	\end{claim}
	Indeed, by \eqref{eq:22}, for every $A > 1$ and $\epsilon > 0$ there is $\delta > 0$ such that for all
	$n \in \NN$, if $0 < a_n, b_n < \delta$ then
	\[
		\psi(a_n) \leq A \psi(b_n)
		\max\big\{
		(a_n/ b_n)^{\frac{\alpha}{2} + \epsilon},
		(a_n/b_n)^{\frac{\alpha}{2} - \epsilon}
		\big\}.
	\]
	Hence,
	\[
		\lim_{n \to \infty} \frac{\psi(a_n)}{\psi(b_n)} \leq A,
	\]
	and since $A > 1$ was arbitrary, the limit is bounded by $1$. Analogously, we show the reverse inequality.

	Now, since
	\[
		\lim_{n \to \infty} \frac{1 - e^{-\tau/t_n}}{\tau/t_n} = 1,
	\]
	by Claim \ref{cl:1}, we get
	\[
		\lim_{n \to \infty} \frac{\psi\big(1 - e^{-\tau/t_n}\big)}{\psi(t_n^{-1})} = \tau^{\alpha/2}.
	\]
	Hence, by \eqref{eq:18} and \eqref{eq:9}
	\begin{align}
		\label{eq:17}
		\lim_{n \to \infty} \Lambda_n (f_{\tau}) =\tau^{\alpha/2 - 2} =\frac{1}{%
		\Gamma(2 - \alpha/2)} \int_0^\infty e^{-\tau x} x^{1-\alpha/2} {\: \rm d}x.
	\end{align}
	Finally, density of the linear span $\calT$ of the set $\{f_\tau : \tau > 0\}$ and equicontinouity of 
	$(\Lambda_n : n \in \NN)$ allows us to extend the limit in \eqref{eq:17} over all of
	$f \in \mathcal{S} \big([0,\infty )\big)$ giving
	\begin{align*}
		\lim_{n \to \infty} \Lambda_n (f) 
		= \frac{1}{\Gamma(2 - \alpha/2)} \int_0^\infty f(x) x^{1-\alpha/2} {\: \rm d}x.
	\end{align*}
	
	For the completeness of our presentation we give a sketch of the proof that
	$\calT$ is dense in $\calS\big([0, \infty)\big)$ (see e.g.
	\cite{vlad, vdz}). By the Hahn--Banach theorem it is enough to show that if
	$\Lambda \in \mathcal{S}^\prime \big([0,\infty )\big)$ and $\Lambda (\phi)=0$ for all
	$\phi \in \calT$, then $\Lambda$ is the zero functional.
	Assume that $\Lambda $ vanishes on $\calT$ and let $\mathcal{L} \Lambda (z)
	=\Lambda (e^{-z\cdot})$, $\Re z >0$, be the Laplace transform of $\Lambda$.
	Since $\Lambda =0$ on $\calT$ we get $\mathcal{L} \Lambda (\lambda )=0$ for $\lambda >0$.
	But the Laplace transform is analytic in the half-plane $\Re z>0$, whence
	$\mathcal{L} \Lambda(z)=0$, for $\Re z>0$. Using the connection between Laplace and
	Fourier transforms, 
	\begin{align*}
		\widehat{\Lambda}(\xi)=\lim_{\lambda \to 0^+}\mathcal{L} \Lambda (\lambda +i\xi),
		\quad \mathrm{in}\ \mathcal{S}^\prime \big([0,\infty )\big),
	\end{align*}
	we obtain that $\widehat{\Lambda}=0$, i.e. $\Lambda $ is the zero functional as desired.

	Next, we claim that 
	\begin{equation}
		\label{eq:10}
		\lim_{n \to \infty} \frac{\calF_n(t_n)}{n t_n \psi\big(t_n^{-1}\big)}
		= \frac{1}{\Gamma(2-\alpha/2)}.
	\end{equation}
	For a fixed $\epsilon > 0$, we choose $\phi_+ \in \mathcal{S}\big([0,\infty )\big)$
	such that $0 \leq \phi_+ \leq 1$ and 
	\begin{equation*}
		\phi_+(x) = 
		\begin{cases}
			1 & \text{ for } 0 \leq x \leq 1, \\ 
			0 & \text{ for } 1+\epsilon \leq x.%
		\end{cases}
	\end{equation*}
	Then we have 
	\begin{align*}
		\calF_n(t_n) = \int_0^{t_n} {\: \rm d} \calF_n(s)
		\leq \int_0^{t_n} \phi_+(s/t_n) {\: \rm d} \calF_n(s) 
		\leq \int_0^\infty \phi_+(s/t_n) {\: \rm d}\calF_n(s).
	\end{align*}
	Hence, by the integration by parts we get 
	\begin{equation*}
		\frac{\calF_n(t_n)}{n t_n \psi\big(t_n^{-1}\big)} \leq -\Lambda_n (\phi_+^{\prime }).
	\end{equation*}
	Therefore, we may estimate 
	\begin{align*}
		\limsup_{n \to \infty} \frac{\calF_n(t_n)}{n t_n \psi\big(t_n^{-1}\big)} 
		&\leq \frac{-1}{\Gamma(2-\alpha/2)} \int_0^{\infty} s^{1 - \alpha/2} \phi^{\prime }_+(s) {\: \rm d} s \\
		&=\frac{1}{\Gamma(1-\alpha/2)} \int_0^{\infty} s^{ - \alpha/2} \phi_+(s) {\: \rm d}s
		\leq \frac{1}{\Gamma(2 - \alpha/2)} (1+\epsilon)^{1-\alpha/2}.
	\end{align*}
	Similarly, by taking $\phi_- \in \mathcal{S}\big([0,\infty )\big)$, $0 \leq \phi_- \leq 1$ satisfying
	\begin{equation*}
		\phi_-(x) = 
		\begin{cases}
			1 & \text{ for } 0 \leq x \leq 1-\epsilon, \\ 
			0 & \text{ for } 1 \leq x,%
		\end{cases}
	\end{equation*}
	one can show that 
	\begin{equation*}
		\liminf_{n \to \infty} \frac{\calF_n(t_n)}{n t_n \psi\big(t_n^{-1}\big)} 
		\geq \frac{1}{\Gamma(2 - \alpha/2)} (1 - \epsilon)^{1 - \alpha/2}.
	\end{equation*}
	Since $\epsilon$ was arbitrary, we conclude \eqref{eq:10}.

	To prove \eqref{eq:11} we follow the line of reasoning from \cite[Theorem 1.7.2]{bgt}.
	Let $\epsilon > 0$. The function $1-F_n(s)$ is non-increasing therefore
	\begin{equation}
		\label{eq:12}
		\calF_n(t) - \calF_n((1-\epsilon) t) = \int_{(1-\epsilon)t}^{t} (1-F_n(s)) {\: \rm d}s 
		\geq \epsilon t (1 - F_n(t))
	\end{equation}
	and 
	\begin{equation}
		\label{eq:12a}
		\calF_n((1+\epsilon)t) - \calF_n(t) = \int^{(1+\epsilon)t}_t (1-F_n(s)) {\: \rm d} s
		\leq \epsilon t (1 - F_n(t)).
	\end{equation}
	By \eqref{eq:10} and regular variation of $\psi$,  
	\begin{equation*}
		\lim_{n \to \infty} \frac{\calF_n ((1-\epsilon) t_n)}{n t_n \psi\big(t_n^{-1}\big)} 	
		=\frac{(1-\epsilon)^{1-\alpha/2}}{\Gamma(2 - \alpha/2)}.
	\end{equation*}
	Hence, \eqref{eq:12} implies 
	\begin{equation*}
		\limsup_{n \to \infty} \frac{1 - F_n(t_n)}{n \psi\big(t_n^{-1}\big)}
		\leq 
		\frac{1}{\Gamma(2 - \alpha/2)} \cdot \frac{1 - (1-\epsilon)^{1-\alpha/2}}{\epsilon}.
	\end{equation*}
	Similarly, by \eqref{eq:12a} we get 
	\begin{equation*}
		\liminf_{n \to \infty} \frac{1 - F_n(t_n)}{n \psi\big(t_n^{-1}\big)} 
		\geq
		\frac{1}{\Gamma(2 - \alpha/2)} \cdot \frac{(1+\epsilon)^{1-\alpha/2} -1}{\epsilon}.
	\end{equation*}
	Finally, by taking $\epsilon$ tending to zero we obtain \eqref{eq:11}.
\end{proof}

\section{The domain of attraction}
Let  $(X_n : n \in \NN)$ be a sequence of independent identically distributed 
random variables with common distribution $\mu$. The distribution $\mu$ belongs to the domain of attraction
of the $\alpha$-stable law if for every $n \in \NN$ there are $a_n \in \RR^d$ and $b_n \in (0, \infty)$ such that
the random variable
\[
	\big( \sum_{j = 1}^n X_j - a_n \big)/b_n
\]
converges in law to $\alpha$-stable random variable. As in the Section \ref{sec:3} we assume that $\psi$
is the Bernstein function regularly varying at zero of index $\alpha/2$, $0 < \alpha < 2$, such that $\psi(0) = 0$
and $\psi(1)=1$. The main result of this subsection is the following theorem.
\begin{theorem}
	\label{domain}
	The law of $S_1^\psi$ belongs to the domain of attraction of the symmetric
	$\alpha$-stable law. 
\end{theorem}
We start with two elementary lemmas. Let $\Phi$ and $\Phi_\psi$ denote characteristic functions of
$p$ and $p_\psi$, respectively.
\begin{lemma}
	\label{newFourier}
	For all $\xi \in \RR^d$,
	\begin{align*}
		\Phi _{\psi}(\xi)=1-\psi (1-\Phi(\xi)).
	\end{align*}
\end{lemma}
\begin{proof}
	Notice that $\Phi(\xi)$ may take complex values, thus we need to use the holomorphic
	extension of the Bernstein function $\psi$. For any $z\in \mathbb{C}$ such
	that $\Re z \geq 0$ we may write (see \cite[Proposition 3.5]{sch}) 
	\begin{align*}
		\psi (z)=a+bz+\int_{[0, \infty)} (1-e^{-zt}) {\: \rm d}\nu (t).
	\end{align*}
	The function $\psi (z)$ is continuous for $\Re z\geq 0$ and
	holomorphic for $\Re z >0$. Thus, for any $z\in \mathbb{C}$ such
	that $\Re z \leq 1$, we have 
	\begin{align*}
		1-\psi (1-z)&= 1-b(1-z)-\int_{[0, \infty)} \big( 1-e^{-t(1-z)} \big) {\: \rm d} \nu (t) \\ 
		&= bz+\int_{[0, \infty)} e^{-t}\sum _{n\geq 1}\frac{t^n z^n}{n!} {\: \rm d} \nu(t) \\
		&= bz+\sum _{n\geq 1} \frac{z^n}{n!}\int _{[0, \infty)} e^{-t}t^n {\: \rm d} \nu(t) \\
		&=\sum _{n\geq 1}c(\psi ,n)\, z^n  
	\end{align*}
	where the coefficients $c(\psi ,n)$ were defined by \eqref{eq:3}. On the other hand 
	\begin{align*}
		\Phi _{\psi} (\xi)=\sum _{x\in \ZZ ^d}p_\psi (x)e^{i\xi x}
		&=\sum_{x\in \ZZ ^d}e^{i\xi x}\sum _{k\geq 1}p(x,k)c(\psi ,k) \\
		&=\sum _{k\geq 1}c(\psi ,k)\big(\Phi(\xi)\big)^{k},
	\end{align*}
	which finishes the proof.
\end{proof}

\begin{lemma}
	\label{psiprop}
	For $u>0$ and $\Re z\geq 0$, 
	\begin{align*}
		\big|\psi \big(u(1+z)\big)-\psi (u)\big|\leq |z|\psi (u).
	\end{align*}
	In particular, we have 
	\begin{align*}
		\psi \big(u(1+i\epsilon)\big)=\psi (u)\big(1+\calO (\epsilon)\big).
	\end{align*}
\end{lemma}
\begin{proof}
From the representation of the Bernstein function we get 
\begin{align*}
	\psi \big(u(1+z)\big)-\psi (u)=buz+\int _{[0, \infty)}e^{-ut}
	\big( 1-e^{-uzt} \big) {\: \rm d} \nu (t).
	\end{align*}
	Since $|1-e^{-z}|\leq |z|$ for $\Re z \geq 0$, we obtain 
	\begin{align*}
		\big|\psi \big(u(1+z)\big)-\psi (u)\big|
		&\leq bu|z| +|z|\int_{[0, \infty)} ute^{-ut} {\: \rm d}\nu (t) \\
		&\leq bu|z|+|z|\int _{[0, \infty)} \big (1-e^{-ut}\big){\: \rm d} \nu (t)
		=|z|\psi (u),
	\end{align*}
	where we used the inequality $e^x \geq 1+x$.
\end{proof}
\begin{proof}[Proof of Theorem \ref{domain}]
	Let $a_n \equiv 0$. For $(b_n : n \in \NN)$ we choose a sequence of real numbers such
	that
	\begin{equation}
		\label{eq:34}
		\lim_{n \to \infty} n \psi \big(b_n^{-2}\big) = 1.
	\end{equation}
	We are going to show that for every $\xi \in \RR^d$,
	\[ 
		\lim_{n \to \infty} \EE\big(e^{i\xi (S_n^{\psi}-a_n)/b_n}\big)
		=\exp\{-2^{-\alpha/2} \sprod{Q \xi}{\xi}^{\alpha /2}\}
	\]
	where $Q$ is defined in \eqref{eq:41}. Since for some $C > 0$ and for all
	$x \in \RR$ 
	\[
		\Big|1 - \cos x - \frac{x^2}{2} \Big| \leq C x^4,
	\]
	then for $\theta \in [-\pi, \pi)^d$ we can estimate
	\begin{align}
		\nonumber
		\Big| \Re\big(1 - \Phi(\theta)\big) - \frac{1}{2} \sprod{Q \theta}{\theta} \Big|
		& \leq
		\sum_{v \in \calV} p(v) \Big|1 - \cos\sprod{v}{\theta} - \frac{1}{2} \sprod{v}{\theta}^2\Big| \\
		\label{eq:26}
		& \leq 
		C_1 \norm{\theta} \sprod{Q \theta}{\theta}.
	\end{align}
	Next, since
	\[
		\sum_{v \in \calV} p(v) \cdot v = 0,
	\]
	and there is $C > 0$ such that for all $x \in \RR$
	\[
		\abs{\sin x - x} \leq C \abs{x}^3,
	\]
	we get
	\begin{align}
		\nonumber
		\abs{\Im \Phi(\theta)} = \Big| \sum_{v \in \calV} p(v) \big(\sin \sprod{v}{\theta} - \sprod{v}{\theta}\big) \Big|
		& 
		\leq 
		\sum_{v \in \calV} p(v) \big| \sin \sprod{v}{\theta} - \sprod{v}{\theta} \big| \\
		\label{eq:43}
		& \leq
		C_2
		\norm{\theta} \sprod{Q\theta}{\theta}.
	\end{align}
	Therefore, by Lemma \ref{psiprop} and estimates \eqref{eq:26} and \eqref{eq:43}, we obtain
	\begin{align}
		\nonumber
		\big|
		\psi\big(1 - \Phi(\theta)\big) - \psi\big(\Re\big(1 - \Phi(\theta)\big)\big)
		\big| 
		& \leq 
		\bigg|\frac{\Im \Phi(\theta)}{\Re\big(1 - \Phi(\theta)\big)} \bigg|
		\psi\big(\Re\big(1 - \Phi(\theta)\big)\big) \\
		& \label{eq:36} \leq
		4 C_2 \norm{\theta}
		\psi\big(\Re\big(1 - \Phi(\theta)\big)\big).
	\end{align}
	provided $\norm{\theta} \leq (4 C_1)^{-1}$. Now, let us fix $\xi \in \RR^d$. We have
	\begin{align*}
		\EE \big(e^{i\xi \, S_n^{\psi}/b_n}\big)
		=\big( \Phi _{\psi} (\xi /b_n) \big)^n
	\end{align*}
	where
	\[
		\Phi_\psi(\xi) = 1 - \psi(1 - \Phi(\xi)).
	\]
	Since
	\[
		\lim_{n \to \infty} \Phi_\psi(\xi/b_n) = 1,
	\]
	there is $N > 0$ such that for $n > N$, $\Re \Phi_\psi(\xi/b_n) > 1/2$.
	Hence, by Lemma \ref{newFourier} and \eqref{eq:36}, we get
	\footnote{$\Log$ denotes the principal value of complex logarithm.}
	\begin{align*}
		\Log \Phi_\psi(\xi / b_n) &= \Log \big(1-\psi \big(1-\Phi (\xi /b_n)\big)\big) \\
		&= - \psi \big(1-\Phi(\xi /b_n)\big) \Big(1 + \calO\big(\big|\psi \big(1-\Phi (\xi /b_n)\big)\big|\big)\Big)\\
		&= - \psi \big(\Re(1 - \Phi(\xi /b_n)\big) 
		\Big(1 + \calO(b_n^{-1}) + \calO\big(\psi \big(\Re (1-\Phi (\xi /b_n))\big)\big)\Big).
	\end{align*}
	Next, by \eqref{eq:26}, we have
	\[
		\lim_{n \to \infty}	
		\frac{\Re(1 - \Phi(\xi/b_n))}{\sprod{Q \xi}{\xi}/(2b_n^2)} = 1,
	\]
	thus, by applying Claim \ref{cl:1}, we obtain
	\[
		\lim_{n \to \infty} \frac{\psi\big(\Re(1 - \Phi(\xi/b_n))\big)}{\psi\big(b_n^{-2}\big)}
		=
		\lim_{n \to \infty} \frac{\psi\big(\sprod{Q \xi}{\xi}/(2b_n^2)\big)} {\psi\big(b_n^{-2}\big)} =
		2^{-\alpha/2} \sprod{Q \xi}{\xi}^{\alpha/2}.
	\]
	Finally, by \eqref{eq:34}, we conclude
	\[
		\lim_{n \to \infty} n \Log \Phi_\psi(\xi /b_n) 
		= -2^{-\alpha/2} \sprod{Q \xi}{\xi}^{\alpha/2}. \qedhere
	\]
\end{proof}

\section{Functional Limit Theorem}\label{sec:FLT}
Without loss of generality we prove the functional limit theorem in the one-dimensional case. For simplicity we consider random walk $(S, p) \in \mathscr{S}$ such that $S_n=\sum _{i=1}^n X_i$, where $X_i$ are independent and identically distributed random variables taking values in $\ZZ $ and satisfying
\begin{align}
 	\EE X_1=0\quad \mathrm{and}\quad \mathrm{Var}(X_1)=1.\label{moments}
\end{align} 
We consider the discrete subordinator 
$(\tau_n : n \in \NN)$ defined by equations \eqref{eq:3} and \eqref{eq:4} and we set $\tau_0 = 0$.
We assume that $\psi$ is a Bernstein function regularly varying at zero of index $\alpha/2$, $0 < \alpha < 2$,
such that $\psi(0) = 0$ and $\psi(1)=1$. 
Let $S^\psi$ be the subordinate random walk and let $B$ be the Brownian motion and $B^\alpha$ the symmetric
$\alpha$-stable process in $\RR^d$. In this section we aim to prove the following theorem.
\begin{theorem}\label{FLT}
	Assume that $(b_n : n \in \NN)$ is a sequence of positive numbers such that
	\[
		\lim _{n\to \infty } n\psi(b_n^{-1}) = 1.
	\]
	Then the sequence of random elements $\big(b_n^{-1/2} S^\psi_{[nt]} : t \geq 0 \big)$ converges in the Skorohod
	space to the random element $B^\alpha$.
\end{theorem}
Let us recall that a sequence $(X^{(n)} : n \in \NN)$ of random elements converges to a random element $X$
in the Skorohod space $\DD\big([0, \infty), \RR\big)$ equipped with the $J_1$-topology if the following two conditions hold
(see \cite[Theorem 16.10 and Theorem 16.11]{Kall}):
\begin{enumerate}
\item
	The finite dimensional distributions of $X^{(n)}$ converge to the distribution of $X$.
\item
	For any bounded sequence $(T_n : n \in \NN)$ of $X^{(n)}$-stopping times,
	and any sequence $(h_n : n \in \NN)$ of positive numbers converging to zero we have
	\[
	\lim_{n \to \infty} \PP\Big(\big| X^{(n)}_{T_n+h_n}-X^{(n)}_{T_n} \big| \geq \epsilon \Big) = 0.
	\]
\end{enumerate}

Our approach to prove Theorem \ref{FLT} is based on the continuity of a composition mapping in the Skorohod space.
More precisely, if a sequence of random elements $(X^{(n)}, \tau^{(n)})$ converges to 
$(X, \tau)$ in the product Skorohod space $\DD\big([0,\infty ), \RR \big)
\times \DD \big([0,\infty ), [0,\infty )\big)$, and if the trajectories of the limiting process $X_t$ are
continuous, and the trajectories of the process $\tau_t$ are non-decreasing then the composition of the processes
\[
	Y^{(n)}_t = X^{(n)}_{\tau^{(n)}_t}
\]
converges to the process $Y_t = X_{\tau_t}$ in $\DD \big([0,\infty ), \RR \big)$, see \cite[Theorem 3.1]{Whitt}
for details. We also use the following invariance principle, for any $S \in \mathscr{S}$ satisfying \eqref{moments}
we have
\begin{align}
	\label{SRW_FLT}
	\lim_{n \to \infty} \frac{1}{\sqrt{n}} S_{[nt]} = B_t
\end{align}
in the Skorohod space $\DD \big([0,\infty), \RR\big)$, see \cite[Theorem 16.1]{Bill}. In view of the above, to prove
Theorem \ref{FLT} we show the convergence of appropriately scaled subordinator $(\tau_n : n \in \NN)$ and then we
combine it with \eqref{SRW_FLT}.

We start with the following elementary lemma.
\begin{lemma}
	\label{lem:1}
	Let $(b_n : n \in \NN)$ be a sequence of positive numbers such that 
	\[
		\lim_{n \to \infty} n \psi\big(b_n^{-1}\big) = 1.
	\]
	Then $\big(b_n^{-1} \tau_n : n \in \NN\big)$ converges in distribution to $(\alpha/2)$-stable law.
\end{lemma}
\begin{proof}
	Let $F_1(t) = \PP(\tau_1 \leq t)$. We observe that the argument given in the proof of
	\cite[Theorem 2(b), Section XIII.6]{feller}, shows that if
	\begin{equation}
		\label{eq:38}
		1-F_1(t) \sim \frac{1}{\Gamma (\alpha /2)} t^{-\alpha/2} \ell(t),
	\end{equation}
	as $t$ tends to infinity, then $b_n^{-1} \tau_n$ converges in distribution to $(\alpha/2)$-stable law. Hence,
	we just need to verify \eqref{eq:38}. Since by \eqref{eq:23}
	\begin{align*}
		\calL\{ 1 - F_1 \}(\lambda) \sim \lambda ^{\alpha /2 -1} \ell (1/\lambda ),
	\end{align*}
	as $\lambda$ tends to zero, the asymptotic \eqref{eq:38} is a consequence of the Tauberian theorem together
	with the monotone density theorem (see e.g. \cite[Theorem 1.7.1 and Theorem 1.7.2]{bgt}).
\end{proof}

For $0 < \beta < 1$, let us denote by $\sigma^\beta = (\sigma^\beta_t : t \geq 0)$ the one-sided $\beta$-stable
subordinator which is uniquely determined by the following relation
\[
	\EE\Big(e^{-\lambda \sigma^\beta_t}\Big) = e^{-s \lambda^\beta}.
\] 
\begin{theorem}\label{cor}
	Let $(b_n : n \in \NN)$ be a sequence of positive numbers such that
	\begin{equation}
		\label{eq:21}
		\lim_{n \to \infty} n \psi\big(b_n^{-1}\big) = 1.
	\end{equation}
	Then the sequence of random elements $\big(b_n^{-1} \tau_{[n t]} : t \geq 0\big)$ converges in
	the Skorohod space $\DD\big([0,\infty ), [0,\infty )\big)$ to the random element $\sigma^{\alpha/2}$.
\end{theorem}
\begin{proof}
	Let $X^{(n)}_t = b_n^{-1} \tau_{[n t]}$. First, we show that the finite dimensional distributions of
	$X^{(n)}_t$ converge to the finite dimensional distribution of $\sigma^{\alpha/2}_t$. Since processes $X^{(n)}$
	have independent increments it is enough to prove convergence of their one dimensional distributions. Let us fix
	$t > 0$. We write 
	\begin{align*}
		X^{(n)}_t = \frac{b_{[nt]}}{b_n} \cdot \frac{1}{b_{[nt]}} \tau_{[n t]}.
	\end{align*}
	Since $\sigma^{\alpha/2}_1$ has $(\alpha/2)$-stable law, by Lemma \ref{lem:1} we obtain that
	$b_{[n t]}^{-1} \tau_{[n t]}$ converges in distribution to $\sigma^{\alpha/2}_1$.
	From the other side, by the uniform convergence theorem for regularly varying functions we get
	\[
		\lim_{n \to \infty} \frac{b_{[nt]}}{b_n} = t^{2/\alpha}.
	\]
	Hence, $\big(X_t^{(n)} : n \in \NN\big)$ converges in distribution to
	$t^{\alpha/2} \sigma^{\alpha/2}_1$. Because $t^{\alpha/2} \sigma^{\alpha/2}_1$ and $\sigma_t^{\alpha/2}$ have
	the same distribution the conclusion follows.
	
	Next, let $(T_n : n \in \NN)$ be any bounded sequence of $X^{(n)}$-stopping times,
	and $(h_n : n \in \NN)$ any sequence of positive numbers converging to zero. We are going to show that
	for any $\epsilon > 0$
	\begin{equation}
		\label{eq:20}
		\lim_{n \to \infty}
		\PP\Big(
		\big|
		X^{(n)}_{T_n + h_n} - X^{(n)}_{T_n}
		\big|
		> \epsilon
		\Big) = 0.
	\end{equation}
	Since
	\[
		X^{(n)}_{T_n + h_n} - X^{(n)}_{T_n} = b_n^{-1}\big(\tau_{[n T_n + n h_n]} - \tau_{[n T_n]}\big),
	\]
	$X^{(n)}_{T_n + h_n} - X^{(n)}_{T_n}$ has the same distribution as $b^{-1}_n \tau_{[n h_n] + j_n}$
	where 
	\[
		j_n = [nT_n + n h_n] - [n T_n] - [n h_n].
	\]
	We notice that $j_n \in \{0, 1\}$, thus
	\begin{align*}
		\PP\big( b_n^{-1} \big(\tau_{[n h_n] + j_n} - \tau_{[n h_n]}\big) > \epsilon\big)
		 =
		\PP\big(\tau_{j_n} > b_n \epsilon\big) 
		 \leq
		\PP\big(\tau_1 > b_n \epsilon),
	\end{align*}
	hence	
	\[
		\lim_{n \to \infty}
		\PP\big( b_n^{-1} \big(\tau_{[n h_n] + j_n} - \tau_{[n h_n]}\big) > \epsilon\big) = 0.
	\]
	We conclude that to prove \eqref{eq:20} it suffices to show that
	\begin{equation}
		\label{eq:30}
		\lim_{n \to \infty}
		\PP\big(b_n^{-1} \tau_{[n h_n]} > \epsilon\big) = 0.
	\end{equation}
	Indeed, suppose the sequence $(n h_n : n \in \NN)$ is bounded by $m \in \NN$. Then
	\[
		\lim_{n \to \infty}
		\PP\big( b_n^{-1} \tau_{[n h_n]} > \epsilon\big)
		\leq
		\lim_{n \to \infty}
		\PP\big( \tau_m > b_n \epsilon\big) = 0.
	\]
	Otherwise, let $(n_j : j \in \NN)$ be any increasing sequence, and $(n_{j_k} : k \in \NN)$ its
	subsequence with the property that $\big([n_{j_k} h_{n_{j_k}}] : k \in \NN\big)$ is strictly increasing.
	We set
	\[
		m_k = [n_{j_k} h_{n_{j_k}}].
	\]
	Observe that by \eqref{eq:21} and a regular variation of $\psi$ we have
	\[
		\lim_{k \to \infty}
		m_k \psi\big(\epsilon^{-1} b_{n_{j_k}}^{-1} \big)
		\leq
		\lim_{k \to \infty}
		h_{n_{j_k}} n_{j_k} \psi\big(\epsilon^{-1} b_{n_{j_k}}^{-1} \big) = 0.
	\]
	Therefore, we can apply Theorem \ref{thm:2} to estimate
	\begin{align*}
		\lim_{k \to \infty}
		\PP\big( b_{n_{j_k}}^{-1} \tau_{m_k} > \epsilon\big) 
		& = 
		\lim_{k \to \infty} 1 - F_{m_k}(b_{n_{j_k}} \epsilon) \\
		& \leq
		C
		\lim_{k \to \infty}
		m_k \psi\big(\epsilon^{-1} b_{n_{j_k}}^{-1}\big) = 0.
	\end{align*}
	Hence,
	\[
		\lim_{n \to \infty}
		\PP\big(b_n^{-1} \tau_{[n h_n]} > \epsilon\big) = 0,
	\]
	which implies \eqref{eq:20}.
\end{proof}

\begin{proof}[Proof of Theorem \ref{FLT}]
	Since $S$ and $\tau$ are independent, we apply \cite[Theorem 3.2]{Bill} and combine it with \eqref{SRW_FLT}, 
	and Theorem \ref{cor} to get that
	\begin{align*}
		\lim_{n \to \infty}
		\big(b_n^{-\frac{1}{2}} S_{[b_n t]},  b_n^{-1} \tau _{[nt]}\big) = 
		\big(B_t, \sigma^{\alpha /2}_t\big)
	\end{align*}
	in the product space $ \DD \big([0,\infty ), \RR \big) \times \DD \big([0,\infty ), [0,\infty )\big)$.
	Moreover, the trajectories of $B$ are continuous whereas the paths of $\sigma ^{\alpha /2}$ are increasing,
	whence we can apply \cite[Theorem 3.1]{Whitt} to get convergence of the related composition,
	\begin{align*}
		\lim_{n \to \infty}
		b_n^{-\frac{1}{2}} S^\psi_{[nt]}
		=B_{\sigma^{\alpha /2}_t}
	\end{align*}
	in $\DD\big([0,\infty ), \RR \big)$. Since $B_{\sigma_t^{\alpha/2}}$ has the same distribution as $B^\alpha_t$,
	the proof is finished.
\end{proof}
\begin{remark}
In paper \cite{mim} author considered the converse statement in Theorem \ref{FLT}, i.e. that convergence in the Skorohod space
implies regular variation of the Bernstein function related to the discrete subordinator. Moreover, the case $\alpha =2$ is treated therein as well.
\end{remark}
\section{Asymptotics of transition functions}
\label{sec:Asymp}
Let $(S, p)$ be a random walk from the class $\mathscr{S}$. We assume that $\psi$
is the Bernstein function regularly varying at zero of index $\alpha/2$, $0 < \alpha < 2$, such that $\psi(0) = 0$
and $\psi(1)=1$. By $S^\psi$ we denote the $\psi$-subordinate random walk
defined in Section \ref{sub:3}. Let $p_\psi(x, n)$ be its transition function. In this section we are going to
show the asymptotic behaviour of $p_\psi(x, n)$ as $n \psi\big(\norm{x}^{-2}\big)$ tends to zero or infinity
while $n$ and $\norm{x}$ tend to infinity.

\subsection{The strong ratio limit theorem}
We start by proving the asymptotic of $p_\psi(0, n)$.
\begin{theorem}
	\label{thm:4}
	Let $p$ be the one-step probability for random walk $S$ in the class $\mathscr{S}$. Then 
	\[
		p_\psi(0, n) \sim D_{d, \alpha} \psi^{-1}\big(n^{-1}\big)^{\frac{d}{2}} 
	\]
	as $n$ tends to infinity, where
	\[
		D_{d, \alpha} = (2 \pi)^{d/2} \frac{\Gamma(1 + d/\alpha)}{\Gamma(1 + d/2)}
		(\det Q)^{-1/2}.
	\]
\end{theorem}
\begin{proof}
	By the Fourier inversion formula and Lemma \ref{newFourier} we can write
	\begin{equation}
		\label{eq:33}
		p_\psi(0, n) = 
		\bigg(\frac{1}{2\pi}\bigg)^d
		\int_{\mathscr{D}_d} \big(1 - \psi(1 - \Phi(\theta))\big)^n \dth
	\end{equation}
	where $\mathscr{D}_d = [-\pi, \pi)^d$. We claim that
	\begin{claim}
		For $z \in \CC$, $\abs{z} \leq 1$,
		\[
			\abs{1 - \psi(1- z)} = 1 \quad \text{if and only if}\quad z = 1.
		\]
	\end{claim}
	Indeed, as in the proof of Lemma \ref{newFourier}, for $z \in \CC$, $\abs{z} \leq 1$ we have
	\[
		1 - \psi(1 - z) = \sum_{k \geq 1} c(\psi, k) z^k,
	\]
	where the series converges absolutely for $\abs{z} \leq 1$. Since
	\[
		\big| 1 - \psi(1-z) \big| \leq \sum_{k \geq 1} c(\psi, k) \abs{z}^k \leq 
		\sum_{k \geq 1} c(\psi, k) = 1,
	\]
	the equation $\abs{1 - \psi(1 - z)} = 1$ implies $\abs{z} = 1$. Now, if $\abs{z_0} = 1$ and satisfies
	$\abs{1 - \psi(1 - z_0)} = 1$ then there is $t \in \RR$ such that
	\[
		\sum_{k \geq 1} c(\psi, k) z_0^k = e^{it},
	\]
	thus
	\[
		z_0^k = e^{i t}
	\]
	for all $k \geq 1$ with $c(\psi, k) > 0$. In particular, $z_0 = z_0^2$, thus $z_0 = 1$.

	Next, by \eqref{eq:22}, there is $x_0 > 0$ such that for all $0 < x, y < x_0$
	\begin{equation}
		\label{eq:32}
		\psi(x) \leq 2 \psi(y) \max\big\{(x/y)^{\alpha/4}, (x/y)^{3\alpha/4} \big\}.
	\end{equation}
	Take $\epsilon > 0$ satisfying
	\[
		\epsilon < \min\Big\{\frac{1}{4 C_1}, \frac{1}{8 C_2}, x_0\Big\}
	\]
	where $C_1$ and $C_2$ are constants from \eqref{eq:26} and \eqref{eq:43}.
	Since $\abs{1 - \psi(1 - \Phi(\theta))} = 1 $ if and only if $\theta \in 2\pi \ZZ^d$, there is $0 < \eta < 1$ such
	that
	\[
		\abs{1 - \psi(1 - \Phi(\theta))} \leq 1 - \eta
	\]
	for all $\theta \in \mathscr{D}_d^\epsilon = \big\{\theta \in\mathscr{D}_d : \norm{\theta} \geq \epsilon \big\}$,
	which implies
	\[
		\bigg(\frac{1}{2\pi}\bigg)^d
		\int_{\mathscr{D}_d^\epsilon} \big| 1 - \psi(1 - \Phi(\theta)) \big|^n \dth
		\leq
		(1 - \eta)^n.
	\]
	Next, we consider the integral over the $\epsilon$-ball centred at the origin. Let $(a_n : n \in \NN)$ be
	a sequence defined by
	\[
		a_n = \sqrt{\psi^{-1}\big(n^{-1}\big)}.
	\]
	By the change of variables we can write
	\[
		a_n^{-d}
		\int_{\norm{\theta} \leq \epsilon}
		\big(1 - \psi(1 - \Phi(\theta))\big)^n \dth
		=
		\int_{\norm{\xi} \leq \epsilon/a_n} 
		\big(1 - \psi(1 - \Phi(a_n \xi))\big)^n {\: \rm d}\xi.
	\]
	We are going to calculate the limit of the integrand for a fixed $\xi \in \RR^d$. By \eqref{eq:26} and \eqref{eq:43},
	we have
	\[
		\lim_{n \to \infty} \frac{1 - \Phi(a_n \xi)}{a_n^2 \sprod{Q \xi}{\xi}} = \frac{1}{2},
	\]
	thus, Claim \ref{cl:1} implies
	\begin{align*}
		\lim_{n \to \infty} n \psi\big(1 - \Phi(a_n \xi)\big)
		& =
		\lim_{n \to \infty} \frac{\psi\big(1 - \Phi(a_n \xi)\big)}{\psi\big(a_n^2 \sprod{Q \xi}{\xi}/2\big)} 
		\cdot
		\frac{\psi\big(a_n^2 \sprod{Q \xi}{\xi}/2\big)}{\psi(a_n^2)} \\
		& = 2^{-\alpha/2} \sprod{Q\xi}{\xi}^{\alpha/2},
	\end{align*}
	because $n^{-1} = \psi(a_n^2)$. Therefore,
	\[
		\lim_{n \to \infty}
		n \Log\big(1 - \psi(1 - \Phi(a_n \xi))\big) = - 2^{-\alpha/2} \sprod{Q \xi}{\xi}^{\alpha/2}.
	\]
	Using polar coordinates, we can calculate
	\[
		\int_{\RR^d} \exp\big(-2^{-\alpha/2} \sprod{Q\xi}{\xi}^{\alpha/2}\big) {\: \rm d}\xi
		=
		(2 \pi)^{d/2}
		\frac{\Gamma(1 + d/\alpha)}{\Gamma(1+d/2)}
		(\det Q)^{-1/2},
	\]
	thus to finish the proof, we need to show that the integrand $\big(1 - \psi(1 - \Phi(a_n \xi))\big)^n$ 
	on $\big\{\xi \in \RR^d : \norm{\xi} \leq \epsilon /a_n\big\}$ has an integrable majorant. Indeed,
	by \eqref{eq:36} and the choice of $\epsilon$,
	\[
		\big| 1 - \psi(1 - \Phi(a_n \xi) ) \big|
		\leq
		1 - \frac{1}{2} \psi\big(\Re(1 - \Phi(a_n \xi))\big).
	\]
	From the other side, by \eqref{eq:32} followed by \eqref{eq:26}, we get
	\[
		\psi\big(\Re(1 - \Phi(a_n \xi))\big)
		\geq
		C n^{-1} \min\big\{\sprod{Q \xi}{\xi}^{\alpha/4}, \sprod{Q \xi}{\xi}^{3\alpha/4}\big\}
	\]
	for some constant $C > 0$. Therefore, if $\norm{\xi} \leq \epsilon/a_n$ then
	\[
		\big|1 - \psi(1 - \Phi(a_n \xi))\big|^n
		\leq
		\exp\Big(-C \min\big\{\sprod{Q \xi}{\xi}^{\alpha/4}, \sprod{Q\xi}{\xi}^{3\alpha/4}\big\}\Big).
	\]
	Since the left-hand side is an integrable function on $\RR^d$, the proof is completed.
\end{proof}

The following Corollary provides a variant of the strong ratio limit theorem.
\begin{corollary}
	\label{cor:2}
	Let $p$ be the one-step probability for random walk $S$ in the class $\mathscr{S}$. Then
	\[
		p_\psi(x, n) \sim p_\psi(0, n)
	\]
	as $n \psi\big(\norm{x}^{-2}\big)$ tends to infinity.
\end{corollary}
\begin{proof}
	By the Fourier inversion formula and Lemma \ref{newFourier}, we may write
	\begin{align*}
		\big| p_\psi(x, n) - p_\psi(0, n) \big|
		& \leq
		\bigg(\frac{1}{2\pi}\bigg)^d
		\int_{\mathscr{D}_d}
		\big|1 - \psi(1 - \Phi(\theta))\big|^n
		\big| e^{-i \sprod{\theta}{x}} - 1\big| \dth \\
		& \leq
		C
		\norm{x}
		\int_{\mathscr{D}_d}
		\big|1 - \psi(1 - \Phi(\theta))\big|^n \norm{\theta} \dth.
	\end{align*}
	Now, essentially the same argument as in the proof of Theorem \ref{thm:4}, shows that there is 
	$C > 0$ such that for all $n \in \NN$ and $x \in \ZZ^d$
	\[
		\big| p_\psi(x, n) - p_\psi(0, n) \big|
		\leq
		C
		\norm{x} \psi^{-1}\big(n^{-1}\big)^{\frac{d+1}{2}},
	\]
	which, together with Theorem \ref{thm:4}, implies
	\[
		\big| p_\psi(x, n) - p_\psi(0, n) \big|
		\leq
		C
		p_\psi(0, n) \norm{x} \psi^{-1}\big(n^{-1}\big)^{\frac{1}{2}}.
	\]
	To conclude the proof we need to show that $\norm{x}^2 \psi^{-1}\big(n^{-1}\big)$ approaches
	zero when $n \psi\big(\norm{x}^{-2}\big)$ tends to infinity. To see this, let us observe that
	$\psi^{-1}$ is regularly varying of index $2/\alpha$ at zero. By \eqref{eq:22}, there is
	$\delta > 0$ such that if $n^{-1} \leq \delta$ and $\psi\big(\norm{x}^{-2}\big) \leq \delta$ then
	\[
		\psi^{-1}\big(n^{-1}\big) \leq 2 \psi^{-1}\big(\psi\big(\norm{x}^{-2}\big)\big) 
		\big(n \psi\big(\norm{x}^{-2}\big)\big)^{-3/\alpha},
	\]
	thus
	\[
		\norm{x}^2 \psi^{-1}\big(n^{-1}\big) \leq 2 \big(n \psi\big(\norm{x}^{-2}\big)\big)^{-3/\alpha},
	\]
	which finishes the proof.
\end{proof}

\subsection{Asymptotic of $p_\psi$ as $n \psi\big(\norm{x}^{-2}\big)$ tends to zero}
\begin{theorem}
	\label{thm:1}
	Let $p$ be the one-step probability for random walk $S$ in the class $\mathscr{S}$ having a period $r$.
	As $n \psi\big(\norm{x}^{-2}\big) $ tends to zero we have
	\footnote{$\delta_0$ is the Dirac mass at the origin.}
	\begin{equation}
		\label{eq:28}
		\big(\delta_0 + p + p^{(2)} + \cdots + p^{(r-1)}\big)*p_\psi(x, n) 
		\sim 
		r C_{d, \alpha} n \norm{x}^{-d} \psi\big(\norm{x}^{-2}\big),
	\end{equation}
	where $\norm{x}^2 = \sprod{Q^{-1} x}{x}$, and 
	\begin{equation*}
		C_{d,\alpha} 
		= \alpha 2^{\alpha/2} \pi^{-d/2-1} (\det Q)^{-1/2}
		\Gamma\Big(\frac{\alpha}{2}\Big) \Gamma\Big(\frac{d + \alpha}{2}\Big)
		\sin\Big(\frac{\pi \alpha}{2}\Big).
	\end{equation*}
\end{theorem}
\begin{proof}
	First, we observe that by Tonelli's theorem
	\[
		p^{(j)} * p_\psi(x, n) = \sum_{k = 1}^\infty p^{(j)}*p(x, k) \PP(\tau_n = k)
		= \sum_{k = 1}^\infty p(x, k+j) \PP(\tau_n = k),
	\]
	thus
	\[
		\big(\delta_0 + p + p^{(2)} + \cdots + p^{(r-1)}\big)*p_\psi(x, n) 
		=
		\sum_{k = 1}^\infty \sum_{j = 0}^{r-1} p(x, k + j) \PP(\tau_n = k).
	\]
	Given $k_0 > r$, we consider a sum
	\[
		\calS_1(x, n) = \sum_{k = 1}^{k_0} \sum_{j=0}^{r-1} p(x, k+j) \PP(\tau_n = k).
	\]
	Using the bounds of Alexopoulos \cite[Theorem 1.8]{alexopulos},
	\[
		p(x, k) \leq C k^{-d/2}\exp \big\{ -4 C_0'\norm{x}^2/k\big\},
	\]
	thus for $k \in \{1, \ldots, k_0+r\}$ we get
	\begin{align*}
		p(x, k) 
		&\leq C \norm{x}^{-d} \exp\big\{- 2 C_0' \norm{x}^2/(k_0+r) \big\}
		\frac{\norm{x}^d}{k^{d/2}} \exp\big\{- 2 C_0' \norm{x}^2/k \big\} \\
		&\leq C' \norm{x}^{-d} \exp\big\{- C_0' \norm{x}^2/k_0 \big\}.
	\end{align*}
	Hence, we can estimate 
	\begin{align*}
		\calS_1(x, n)
		& \leq
		C'' \norm{x}^{-d} \exp\big\{- C_0' \norm{x}^2/k_0 \big\}
		\sum_{k = 1}^\infty \PP(\tau_n = k).
	\end{align*}
	For $x \in \ZZ^d$ and $n \in \NN$, we define
	\[
		k_0 = \bigg[
		- \frac{\min\{C_0', 2\}}{4} \cdot \frac{\norm{x}^2}{\log \big\{n \psi\big(\norm{x}^{-2} \big) \big\}}
		\bigg].
	\]
	Then we have
	\begin{equation}
		\label{eq:15}
		\calS_1(x, n) = o\Big(n \norm{x}^{-d} \psi\big(\norm{x}^{-2}\big)\Big).
	\end{equation}
	\begin{claim}
		\label{cl:2}
		There are $C > 0$ and $R > 0$ such that for all $n \in \NN$ and $x \in \ZZ^d$, if $\norm{x} \geq R$ then
		\begin{align}
			\label{eq:24a}
			k_0 &\geq C \norm{x}^{7/4},\\
			\label{eq:24b}
			n \psi\big(k_0^{-1}\big) &\leq C \big(n \psi\big(\norm{x}^{-2}\big)\big)^{(1-\alpha/2)/2}.
		\end{align}
	\end{claim}

	Assuming for a moment the validity of Claim \ref{cl:2}, we proceed with the proof of the theorem. We need to find
	the asymptotic of a sum
	\begin{equation*}
		\calS_2(x, n) = \sum_{k >k_0} \sum_{j = 0}^{r-1} p(x, k+j) \PP(\tau_n = k).
	\end{equation*}
 	To do so we want to use the asymptotic from Corollary \ref{cor:1}. Since \eqref{eq:24a} implies that for $k \geq k_0$,
	\[
		\norm{x} \leq C k^{4/7},
	\]
	and
	\[
		k^{-2} \norm{x}^3 \leq C k^{-2/7},
	\]
	there is $C > 0$ such that
	\begin{multline*}
		\Big\lvert
		\sum_{j = 0}^{r-1}
		p(x, k+j)
		- 
		r (2 \pi)^{-d/2} (\det Q)^{-1/2} k^{-d/2} \exp\{-\norm{x}^2/(2k)\}
		\Big\rvert \\
		\leq
		C
		k^{-d/2-2/7}
		\exp\{-\norm{x}^2 /(2k)\}
	\end{multline*}
	for all $k > k_0$. Then, by setting
	\begin{equation*}
		\calI(x, n) = 
		r (2 \pi)^{-d/2} (\det Q)^{-1/2}
		\sum_{k \geq k_0}
		k^{-d/2} \exp\{-\norm{x}^2/(2k)\} \PP(\tau_n = k),
	\end{equation*}
	we obtain
	\begin{align}
		\label{eq:14}
		\big\lvert \calS_2(x, n) - \calI(x, n) \big\rvert 
		& \leq C \sum_{k \geq k_0}
		k^{-d/2-2/7} \exp\{-\norm{x}^2/(2k)\} \PP(\tau_n = k)  \notag \\
		& \leq C_1 k_0^{-2/7}
		\calI(x, n). 
	\end{align}
	Next, for $G_n(t) = 1-F_n(t) = \PP (\tau_n > t)$ we can write
	\begin{equation}
		\label{eq:5}
		\calI(x, n) 
		= 
		-r
		(2 \pi)^{-d/2} (\det Q)^{-1/2}
		\int_{[k_0, \infty)} t^{-d/2} \exp\{-\norm{x}^2/(2t)\} {\: \rm d} G_n(t).
	\end{equation}
	Let 
	\begin{align*}
		I(x, n)=r (2 \pi)^{-d/2} (\det Q)^{-1/2}
		\int_{k_0}^\infty
		\frac{\rm d}{{\rm d}t}
		\Big(t^{-d/2} \exp\{-\norm{x}^2/(2t)\} \Big) G_n(t) {\: \rm d}t.
	\end{align*}
	Since in \eqref{eq:5} the integrand is continuous and $G_n$ is monotone, the integration by parts yields
	\begin{align}
		\label{eq:16}
		\big\lvert \calI(x, n) - I(x, n) \big\rvert 
		& \leq C_1 k_0^{-d/2}\exp\{-\norm{x}^2 /(2 k_0)\}  \notag \\
		&=o\Big(\norm{x}^{-d} n\psi \big(\norm{x}^{-2}\big)\Big),
	\end{align}
	where in the last step we have used the definition of $k_0$. To find the asymptotic behaviour of $I(n,x)$ we compute
	\begin{multline}
		\label{eq:29}
		\frac{\rm d}{{\rm d}t}
		\Big(t^{-d/2}\exp\{-\norm{x}^2 / (2t) \}\Big) \\
		=-\frac{d}{2} t^{-d/2-1}\exp\{-\norm{x}^2/(2t)\} + \frac{\norm{x}^2}{2} t^{-d/2 - 2}\exp\{-\norm{x}^2/(2t)\}.
	\end{multline}
	Next, let us define two integrals
	\begin{align*}
		J_1 &= \frac{d}{2} \int_{k_0}^\infty t^{-d/2- 1} \exp\{-\norm{x}^2/(2t) \}G_n(t) {\: \rm d} t,\\
		J_2 &= \frac{\norm{x}^2}{2} \int_{k_0}^\infty t^{-d/2 -2} \exp\{-\norm{x}^2/(2t)\} G_n(t) {\: \rm d}t.
	\end{align*}
	Then, by \eqref{eq:29}, we have
	\[
		I(x, n) = r (2\pi)^{-d/2} (\det Q)^{-1/2} \big(J_2 - J_1\big). 
	\]
	To show the asymptotic of $J_1$ and $J_2$ we need the following lemma.
	\begin{lemma}
		\label{lem:3} 
		For any $\beta >1$
		\begin{multline*}
			\int_{k_0}^{\infty} t^{-\beta} \ell(t) \exp\{-\norm{x}^2/(2t) \} G_n(t) {\: \rm d}t \\ 
			\sim 
			2^{\beta+\alpha/2-1} \frac{\Gamma(\beta+\alpha/2-1)}{\Gamma(1 - \alpha/2)}
			n \norm{x}^{-2(\beta-1)} \psi\big(\norm{x}^{-2}\big),
		\end{multline*}
		as $n \psi\big(\norm{x}^{-2}\big)$ tends to zero.
	\end{lemma}
	\begin{proof}[Proof of Lemma \ref{lem:3}]
		First, we change variables by setting $u = \norm{x}^{-2}t$
		\begin{multline*}
			\frac{1}{n \norm{x}^{-2(\beta-1)}\psi\big(\norm{x}^{-2}\big)}
			\int_{k_0}^{\infty} t^{-\beta} \exp\{-\norm{x}^2/(2t)\} G_n(t) {\: \rm d}t \\
			= 
			\int_{\norm{x}^{-2} k_0}^\infty
			u^{-\beta}  
			\frac{G_n\big(\norm{x}^2 u\big)}{n \psi\big(\norm{x}^{-2} u^{-1}\big)}
			\frac{\psi\big(\norm{x}^{-2} u^{-1}\big)}{\psi\big(\norm{x}^{-2}\big)}
			\exp\{-1/(2u)\} {\: \rm d}u.
		\end{multline*}
		To estimate the integrand, by \eqref{eq:24b} and Theorem \ref{thm:2}, we have
		\[
			\frac{G_n\big(\norm{x}^2 u\big)}{n \psi\big(\norm{x}^{-2} u^{-1} \big)}
			\leq
			C,
		\]
		and, by \eqref{eq:22}
		\[
			\frac{\psi\big(\norm{x}^{-2} u^{-1} \big)}{\psi\big(\norm{x}^{-2}\big)} 
			\leq C u^{-\alpha/2}
			\max\big\{u^{-\frac{\beta -1}{2}}, u^{\frac{\beta -1}{2}}\big\},
		\]
		thus by applying the dominated convergence we obtain  
		\begin{multline*}
			\int_{\norm{x}^{-2} k_0}^\infty
			u^{-\beta} 
			\frac{G_n\big(\norm{x}^2 u\big)}{n \psi\big(\norm{x}^{-2} u^{-1}\big)}
            \frac{\psi\big(\norm{x}^{-2} u^{-1}\big)}{\psi\big(\norm{x}^{-2}\big)}
			\exp\{-1/(2u)\}
			{\: \rm d}u \\
			\sim
			\frac{1}{\Gamma(1- \alpha/2)}
			\int_0^\infty u^{-\beta-\alpha/2} \exp\{-1/(2u)\} {\: \rm d}u,
		\end{multline*}
		as $n \psi\big(\norm{x}^{-2}\big)$ tends to zero, and the proof of the lemma is finished.
		\end{proof}
		We now return to showing the asymptotic of $I(x, n)$. For $J_1$, by applying Lemma \ref{lem:3}, we obtain
		\begin{equation}
			\label{eq:27a}
			J_1 
			\sim 
			\frac{d}{2} 2^{(d+\alpha)/2} \frac{ \Gamma\left((d+\alpha)/2\right)}
			{\Gamma (1-\alpha /2)} n \norm{x}^{-d} \psi \big(\norm{x}^{-2}\big),
		\end{equation}
		as $n \psi \big(\norm{x}^{-2}\big)$ tends to zero. In the case of $J_2$, we obtain
		\begin{equation}
			\label{eq:27b}
			J_2
			\sim
			2^{(d+\alpha)/2} \frac{ \Gamma\left((d+\alpha+2)/2\right)}
			{\Gamma (1-\alpha /2)} n \norm{x}^{-d} \psi \big(\norm{x}^{-2}\big),
		\end{equation}
		as $n \psi \big(\norm{x}^{-2}\big)$ approaches zero. Putting \eqref{eq:27a} and \eqref{eq:27b} together
		we conclude
		\begin{equation}
			\label{eq:27}
			I(x, n) \sim r C_{d, \alpha} 
			n \norm{x}^{-d} \psi \big(\norm{x}^{-2}\big).
		\end{equation}
		Now, by \eqref{eq:14} and \eqref{eq:16},
		\[
			\calS_2(x, n) = \Big(I(x, n) + o\Big(n\norm{x}^{-d} \psi\big(\norm{x}^{-2}\big)\Big)\Big)\big(1 + o(1)\big),
		\]
		thus \eqref{eq:27} yields
		\[
			\calS_2(x, n) = n \norm{x}^{-d} \psi\big(\norm{x}^{-2}\big) \big(r C_{d,\alpha} + o(1)\big),
		\]
		what together with \eqref{eq:15} gives the desired asymptotic \eqref{eq:28}.
		
		To finish the proof of the theorem we need to justify Claim \ref{cl:2}. First, by \eqref{eq:22}, there is
		$R > 0$ such that for $\norm{x} \geq R$ and $n \in \NN$
		\[
			n \psi\big(\norm{x}^{-2} \big) \geq \psi\big(\norm{x}^{-2}\big) \geq \norm{x}^{-3\alpha}. 
		\]
		Hence,
		\[
			-\log \big\{n \psi\big(\norm{x}^{-2}\big)\big\}
			\leq
			3 \alpha \log \norm{x},
		\]
		and therefore
		\[
			k_0 \geq C \norm{x}^2/\log\norm{x},
		\]
		which easily implies \eqref{eq:24a}.
		
		For the proof of (ii), again using \eqref{eq:22}, there are $R > 0$
		and $C > 1$ such that for all $n \in \NN$ and $\norm{x} \geq R$
		\[
			\ell(k_0) \leq C \ell\big(\norm{x}^2\big) 
			\Big(-\log \big\{n \psi\big(\norm{x}^{-2} \big)\big\}\Big)^{(1-\alpha/2)/2}.
		\]
		Therefore we obtain
		\begin{align*}
			n \psi\big(k_0^{-1}\big) = n k_0^{-\alpha/2} \ell(k_0)
			&\leq
			C\, n\psi\big(\norm{x}^{-2}\big)
			\Big(-\log \big\{n \psi\big(\norm{x}^{-2} \big) \big\}\Big)^{(1+\alpha/2)/2}\\
			&\leq 
			C\, \left(n \psi\big(\norm{x}^{-2} \big)\right)^{(1-\alpha/2)/2},
		\end{align*}
		and the proof of Claim \ref{cl:2} is finished.
\end{proof}

\begin{remark}
	One possible application of Theorem \ref{thm:1} would be to study the Green function $G_\psi$
	of the subordinate random walk
		\[ 
			G_\psi (x) = \sum _{n\geq 0}p_\psi (x,n).
		\]
	We have
	\[
		G_\psi(x) = \sum_{k \geq 1} p(x, k) C(k)
	\]
	where $(C(k): k\geq 0)$ satisfies  
	\[
		C(0)=1\quad \mathrm{and}\quad C(k) = \sum _{n=0}^k c(\psi ,n)C(k-n),\ \ k\geq 1.
	\]
	Observe that the asymptotic formula of Theorem \ref{thm:1} does not provide enough information to obtain 
	asymptotic formula for the Green function $G_\psi(x)$ at infinity. One way to study asymptotic decay of $G_\psi$
	at infinity would be to study asymptotic properties of the sequence $(C(k) : k \geq 0)$. 
	Under certain special assumptions on $\psi $ this problem was solved in the recent paper of Bendikov and Cygan 	
	 \cite{cygan}. In this connection we also mention the closely related paper of Williamson \cite{williamson}.
\end{remark}
\section*{Acknowledgements}
We owe to thank Konrad Kolesko and Stanislav Molchanov for helpful and stimulating discussions. We thank
anonymous referee for helpful comments and suggestions.

\begin{bibliography}{random_walks}
	\bibliographystyle{amsplain}

\providecommand{\bysame}{\leavevmode\hbox to3em{\hrulefill}\thinspace}
\providecommand{\MR}{\relax\ifhmode\unskip\space\fi MR }
\providecommand{\MRhref}[2]{%
  \href{http://www.ams.org/mathscinet-getitem?mr=#1}{#2}
}
\providecommand{\href}[2]{#2}
\begin{thebibliography}{10}

\bibitem{alexopulos}
G.K. Alexopoulos, \emph{Random walks on discrete groups of polynomial volume
  growth}, Ann. Probab. \textbf{30} (2002), no.~2, 723--801.

\bibitem{b}
A.~Bendikov, \emph{Asymptotic formula for symmetric stable semigroups}, Exp.
  Math. (1994), no.~12, 381--384.

\bibitem{cygan}
A.~Bendikov and W.~Cygan, \emph{Alpha-stable random walk has massive thorns},
  Colloq. Math. \textbf{138} (2015), 105--129.

\bibitem{bs}
A.~Bendikov and L.~Saloff-Coste, \emph{Random walks on groups and discrete
  subordination}, Math. Nachr. (2012), no.~285, 580--605.

\bibitem{Bill}
P.~Billingsley, \emph{Convergence of {P}robability {M}easures}, John Wiley \&
  Sons, 1968.

\bibitem{bgt}
N.H. Bingham, Dr~C.M. Goldie, and J.L. Teugels, \emph{Regular variation},
  Encyclopedia of Mathematics and its Applications, vol.~27, Cambridge
  University Press, 1989.

\bibitem{blumen}
R.~M. Blumenthal and R.~K. Getoor, \emph{Some theorems on stable processes}, T.
  Am. Math. Soc. \textbf{95} (1960), 263--273.

\bibitem{bochner}
S.~Bochner, \emph{Diffusion equation and stochastic processes}, Proc. Natl.
  Acad. Sci. U.S.A. \textbf{35} (1949), 368–--370.

\bibitem{doney}
A.R. Doney, \emph{One-sided local large deviation and renewal theorems in the
  case of inifinite mean}, Probab. Theory Rel. \textbf{4} (1997), no.~107,
  451--465.

\bibitem{feller}
W.~Feller, \emph{An introduction to probability theory and its applications},
  vol.~II, John Wiley \& Sons, 1965.

\bibitem{Kall}
O.~Kallenberg, \emph{Foundations of {M}odern {P}robability}, Springer, 2002.

\bibitem{law}
G.~F Lawler and V.~Limic, \emph{{R}andom walks: A modern introduction},
  Cambridge Studies in Advanced Mathematics, Cambridge University Press, 2010.

\bibitem{mim}
A.~Mimica, \emph{On subordinate random walks}, To appear in Forum Math.,
  arXiv:1510.05215, 2016.

\bibitem{nag}
S.V. Nagayev, \emph{On the asymptotic behaviour of one-sided large deviation
  probabilities}, Theory Probab. Appl. \textbf{2} (1980), no.~26, 362--366.

\bibitem{polya}
G.~P\'{o}lya, \emph{On the zeros of an integral function represented by
  {Fourier's} integral}, Messenger of Math. \textbf{52} (1923), 185--188.

\bibitem{pot}
H.S.A. Potter, \emph{The mean values of certain {D}irichlet series, {II}}, P.
  London Math. Soc. \textbf{2} (1942), no.~1, 1--19.

\bibitem{sch}
R.~Schilling, R.~Song, and Z.~Vondra\v{c}ek, \emph{Bernstein functions}, De
  Gruyter Studies in Mathematics, vol.~37, Walter de Gruyter \& Co., 2010.

\bibitem{tr1}
B.~Trojan, \emph{Long time behaviour of random walks on the integer lattice},
  arXiv:1512.09035, 2015.

\bibitem{vlad}
V.S. Vladimirov, \emph{Methods of the theory of generalized functions},
  Analytical Methods and Special Functions, vol.~6, Taylor and Francis, 2002.

\bibitem{vdz}
V.S. Vladimirov, Yu.N. Drozzinov, and B.I. Zavialov, \emph{Tauberian theorems
  for generalized functions}, Kluwer Academic Publishers, 1988.

\bibitem{Whitt}
W.~Whitt, \emph{Some useful functions for functional limit theorems}, Math.
  Oper. Res. \textbf{5} (1980), no.~1, 67--85.

\bibitem{williamson}
J.A. Williamson, \emph{Random walks and {R}iesz kernels}, Pacific J. Math.
  \textbf{25} (1968), 393--415.

\end{thebibliography}
\end{bibliography}
\end{document}